\documentclass[onecolumn]{autart}    
\usepackage[latin1]{inputenc}
\usepackage{amsmath}
\usepackage{amsfonts}
\usepackage{amssymb}
\usepackage{graphicx}
\usepackage{epsfig,psfrag}
\usepackage{xcolor,soul}
\usepackage[mathscr]{euscript}

\usepackage{tikz,circuitikz}
\graphicspath{{figs/}}

\newtheorem{theorem}{Theorem}
\newtheorem{lemma}{Lemma}
\newtheorem{corollary}{Corollary}
\newtheorem{assumption}{Assumption}
\newtheorem{definition}{Definition}
\newtheorem{proof}{Proof}
\newtheorem{example}{Example}
\newtheorem{remark}{Remark}

\renewenvironment{proof}{\begin{trivlist} \item[{ \bf Proof:}] }
	{~\hfill$\Box$ \end{trivlist} }

\newcommand{\He}{{\rm He}}
\newcommand{\rank}{{\rm rank}}
\newcommand{\R}{\mathbb{R}}
\newcommand{\B}{\mathcal{B}}
\newcommand{\K}{\mathbb{K}}
\newcommand{\Vb}{\bar{V}}
\newcommand{\Vp}{V_\perp}
\newcommand{\Pp}{\mathbb{P}}

\begin{document}
	
	\begin{frontmatter}

		\title{Stabilization of rank-deficient continuous-time switched affine systems} 
		\thanks[footnoteinfo]{Corresponding author Lucas N. Egidio. }
		
		\author[UCL]{Lucas N. Egidio}\ead{lucas.egidio@uclouvain.be},    
		\author[FEM]{Grace S. Deaecto}\ead{grace@fem.unicamp.br},    
		\author[UCL]{Rapha\"{e}l Jungers}\ead{raphael.jungers@uclouvain.be}               
		
		\address[UCL]{ICTEAM Institute, Universit\'{e} Catholique de Louvain, 1348, Louvain-la-Neuve, Belgium}  
		\address[FEM]{DMC - School of Mechanical Engineering, UNICAMP,\\ 13083-860, Campinas, SP, Brazil }  
		

		\begin{abstract}                         
			This paper treats the global stabilization problem of continuous-time switched affine systems {that have rank-deficient convex combinations of their dynamic matrices}. {For these systems}, the already known set of attainable equilibrium points has higher dimensionality than in the {full-rank case} due to the existence of what we define as singular equilibrium points. Our main goal is to design a state-dependent switching function to ensure global asymptotic stability of a chosen point inside this set with conditions expressed in terms of linear matrix inequalities. {For this class of systems, global exponential stability is generally impossible to be guaranteed.} Hence, the proposed switching function is shown to ensure {global asymptotic and} local exponential stability of the desired equilibrium point. The position control and the velocity control with integral action of a dc motor driven by an h-bridge fed via a boost converter are used for validation. This practical application example is composed of eight subsystems, and all possible convex combinations of the dynamic matrices are singular.
		\end{abstract}
		\vspace{-0.5cm}
	\end{frontmatter}
	
	\section{Introduction}	
	Switched affine systems constitute a generalization of switched linear systems where the existence of affine terms in the dynamic equations lets the system switch among different dynamic behaviors with distinct equilibria. Therefore, a switched affine system can be briefly defined as a family of affine vector fields, referred to as subsystems, along with a switching signal that selects one of them as active at each instant of time to define the time evolution of the state trajectories. For classical references on switched and hybrid systems, we refer to~\cite{liberzon_switching_2003,decarlo2000perspectives,sun2011stability}. 
	
	Among other particularities, the stabilization problem for switched affine systems has an additional difficulty compared to their linear counterpart: by properly orchestrating the switching, the presence of affine terms in the dynamic equations may allow the system to be stabilized to points in the state space that are not equilibrium points for any of the subsystems. Consequentially, the stabilization problem consists in finding a switching rule that not only yields exclusively stable trajectories but also guarantees asymptotic stability of a given point in the state-space, chosen by the designer. The stabilization of switched affine systems to desired references is the key point that allows, for instance, power converters to obtain steady outputs by switching among topologies that cannot generate these same outputs alone. See~\cite{deaecto2010switched} for some applications of switched affine systems in power electronics.
	
	In the continuous-time domain, when state-dependent switching functions are considered, asymptotic stability is generally linked to the existence of Filippov solutions (see~\cite{liberzon_switching_2003}), which allows the system trajectories to evolve according to an averaged dynamics over those of the available subsystems. Most stabilization results, such as~\cite{deaecto2010switched,bolzern2004quadratic,beneux2019adaptive,sanchez2019practical,hetel2013robust}, provide {sufficient} design conditions that impose the following constraint: there must exist a Hurwitz stable matrix that is a convex combination of the dynamic matrices of the subsystems. 
	
	{A few} works in the literature, however, have relaxed this condition {to some extent, allowing a larger class of switched affine systems to be stabilized}. In~\cite{hetel2014local}, the authors proposed a switching function that locally stabilizes switched affine systems without requiring a stable convex combination of the {dynamics of the subsystems}. In~\cite{scharlau2014switching}, based on a \textit{max-type} Lyapunov function, the authors presented a design procedure based on nonlinear inequalities of the Lyapunov function parameters, which also {avoids this constraint}. The case where integrators are appended to the switched affine system to improve performance under uncertainties was investigated in~\cite{beneux2018integral}, resulting in an augmented switched affine system with all dynamic matrices sharing a non-trivial nullspace and, therefore, having no Hurwitz stable convex combination of the {matrices of the subsystems}. In the same direction, the authors of~\cite{ndoye2019robust} provided a robust relay control with integral action for a buck converter ensuring local stability and null steady-state error of the output voltage without requiring a stable convex combination of the {dynamics of the subsystems}. 
	
	{In this paper, we develop design conditions that go beyond these works since {our conditions} ensure global asymptotic stability and not local as in~\cite{hetel2014local}. Moreover, they are based on linear matrix inequalities and, therefore, simpler to solve than the nonlinear inequalities of~\cite{scharlau2014switching}, and include the integral control proposed in~\cite{beneux2018integral} and~\cite{ndoye2019robust} as particular cases. These design conditions are obtained by investigating} the global asymptotic stabilization of what we define as \textit{singular equilibrium points}. These equilibrium points are related to a convex combination of the {matrices of the subsystems} that is rank-deficient and, therefore, are not linked to a Hurwitz stable matrix. We show that, associated with each of these particular dynamics, there may exist infinitely many equilibrium points that are located in a subspace of the state space. The main contribution of this paper is to provide  a set of linear matrix inequalities (LMIs) together with conditions on the affine terms	to design a switching rule able to ensure global asymptotic stability of a singular equilibrium point chosen by the designer. The conditions on the affine terms were not fully explored by the literature to date, although they are essential to ensure that the time derivative of the Lyapunov function is strictly negative definite. The proposed results {allow us} to globally stabilize the class of systems studied in~{\cite{beneux2018integral,ndoye2019robust}}, as {the structure with an integrator considered therein} represents a particular case in our formulation. Moreover, examples for which the conditions in~\cite{scharlau2014switching} fail to propose a solution can be handled directly without difficulty by the proposed control methodology. Additionally, we present a discussion on the convergence rate for this class of systems, where it is {argued} that \textit{global exponential} stability is impossible to be ensured in the general case. However, our results ensure that exponential convergence holds \textit{locally} around the chosen equilibrium point and we devise a method based on sum-of-squares programming to estimate the exponential convergence rate. {Illustrative examples highlight important aspects of the theory and a practical application concerning the position and velocity control with integral action of a dc motor demonstrates its usefulness and efficiency.}

	\textbf{Notation:} The notation used throughout {this paper} is standard. For real vectors or matrices, $(')$ refers to their transpose. For symmetric matrices, $(\bullet)$ denotes each of their symmetric blocks. 
	The sets of real and natural numbers are $\mathbb{R}$ and $\mathbb{N}$, respectively. 
	The set $\K=\{1, \cdots, N\}$ is composed of the $N$ first positive natural numbers. For any symmetric matrix, $X>(<)~0$ denotes a positive (negative) definite matrix. The trace of a matrix $X$ is denoted as ${\rm  tr}(X)$. The unit simplex is defined as $\Lambda = \{\lambda\in \mathbb{R}^N:\lambda_i\geq 0, \sum_{i \in \K} \lambda_i =1\}$. The convex combination of matrices $X_i, \forall i\in \K,$ is $X_\lambda = \sum_{i\in\K}\lambda_i X_i$, $\lambda\in \Lambda$. For a square matrix $X$, we define the operator $\He(X):=X+X'$. {The minimum and the maximum eigenvalues of a {symmetric real} matrix $X$ are $s_{\min}(X)$ and $s_{\max}(X)$, respectively. A matrix $[v_i]_{i\in \mathcal{I}}$ is constructed with column vectors $v_i$ by taking indices $i$ in increasing order from the discrete set $\mathcal{I}$.}

	\section{Problem formulation}
	Consider a switched affine system defined as
	\begin{equation}
	\dot{x}(t) = A_{\sigma(t)}x(t) +b_{\sigma(t)},\quad x(0)=x_0\label{eq:sys}
	\end{equation}
	where $x(t)\in\R^n$ is the state vector at a time instant $t\geq0$ and the pairs $(A_i,b_i),~i\in\K,$ define $N$ available subsystems. {The switching signal $\sigma(t)\in\K$ selects at each $t\geq0$ one of these subsystems as active and is the only input signal to the system.}
	Our goal is to design a switching rule $\sigma(t) = u(x(t))$  to ensure global asymptotic stability of an equilibrium point $x_e$ chosen inside the set 
	\begin{equation}
	X_e = \{x_e\in\R^n ~:~  A_\lambda x_e+b_\lambda=0, ~\lambda\in\Lambda \}\label{eq:xe_general}
	\end{equation}
	which contains all possible equilibria of~\eqref{eq:sys} when considering solutions $x(t)$ in the sense of Filippov~\cite[p.~13]{liberzon_switching_2003}.
	If for some $\lambda\in\Lambda$ we have that {the convex combination $A_\lambda$ is Hurwitz}, then there exists a correspondence from $\lambda$ to a single $x_e\in X_e$ given by $x_e = -A_\lambda^{-1}b_\lambda$ and many works in the literature have addressed global stabilization problem under this assumption~\cite{deaecto2010switched,bolzern2004quadratic,beneux2019adaptive,sanchez2019practical,hetel2013robust}. However, this problem is understudied when $A_\lambda$ is rank deficient for a given $\lambda\in\Lambda$. In this case, such $\lambda$ is either associated with no equilibrium point or with an equilibrium space, i.e., there exist infinitely many points $x_e\in X_e$ satisfying the linear equation in~\eqref{eq:xe_general}. For the latter case, we formally define these equilibria as follows. 
	
	\begin{definition}\label{def:singular_equilibrium}
		A point $x_e\in\R^n$ is called a \textit{singular equilibrium point} of the system~\eqref{eq:sys} if it satisfies
		$
		A_\lambda x_e + b_\lambda = 0 
		$
		{for} at least one $\lambda \in\Lambda$ such that $\det(A_\lambda)=0$.
	\end{definition}

	{Despite its connection with the results in~\cite{hetel2014local,scharlau2014switching,beneux2018integral,ndoye2019robust}, the global asymptotic stability of singular equilibrium points has received no prior particular attention from the scientific community to the best of our knowledge.  Therefore, its relation with these previous works will be discussed throughout this paper.}

	For some given $\lambda\in\Lambda$, assume that $\mathcal{N}=\{q\in\R^n ~:~A_\lambda q=0\}$ is the non-trivial nullspace of $A_\lambda$ with ${\rm dim}(\mathcal{N})=m>0$. Clearly, $A_\lambda$ is not a {\em Hurwitz} stable matrix nor a regular one. Consider matrices ${V}_\perp\in\R^{n \times m}$ and $\Vb\in\R^{n \times p}$, being $p=n-m$, where ${V}_\perp$ has columns forming an orthonormal basis for $\mathcal{N}$ and $\Vb$ has orthonormal columns that span the orthogonal complement of $\mathcal{N}$, i.e., the row space of $A_\lambda$. Consequently, the following identities are verified
	\begin{equation}
	\Vb'\Vp = 0,~~\Vb'\Vb=I,~~ {\Vp'}\Vp=I.\label{eq:identities}
	\end{equation}
	{These matrices can be efficiently obtained by, for instance, the singular value decomposition of $A_\lambda$.}
	Thus, a singular equilibrium point $x_e$ associated with $\lambda\in\Lambda$ can be decomposed as 
	\begin{equation}\label{eq:xe_singular1}
	x_e = \Vb\bar{x}_e + \Vp x_{e\perp}
	\end{equation}
	with $\bar{x}_e\in\R^{p}$ solution to the equation
	\begin{equation}
	A_\lambda \Vb\bar{x}_e+b_\lambda = 0,\label{eq:xe_singular2}
	\end{equation}
	and an arbitrary $x_{e\perp}\in\R^m$.  Whenever this equation is satisfied, the arbitrariness of $x_{e\perp}$ spans a set of singular equilibrium points associated to $\lambda$, which form an affine subspace in the state space passing through $\Vb\bar{x}_e$. It is important to remark that the expression in~\eqref{eq:xe_singular2} is developed from the necessary condition for attainable equilibrium points represented by the set defined in~\eqref{eq:xe_general}. Thus, even though~\eqref{eq:xe_singular2} defines an overdetermined system of equations, it is no more conservative than the classical necessary condition $A_\lambda x_e+b_\lambda=0$ for equilibrium points of switched affine systems. In other words, for a given $\lambda\in\Lambda$ such that $A_\lambda$ is singular, any solution $x_e$ to $A_\lambda x_e +b_\lambda=0$ can be decomposed as in~\eqref{eq:xe_singular1} to satisfy~\eqref{eq:xe_singular2}.

	{A relevant question within the scope of this work is whether the matrix $\Vb'A_\lambda \Vb$ is singular, as it will be important to obtain our design conditions. The following lemma provides an answer to this question.}

	\begin{lemma} \label{lemma:defective}
		Let a singular matrix $X\in\R^{n\times n}$ along with matrices  $\Vp\in\R^{n\times m}$ and $\Vb\in\R^{n\times p}$ where $p=n-m$ formed by orthonormal vectors spanning the nullspace and row space of $X$, respectively. The two statements are equivalent:
		\begin{enumerate}
			\item[(a)] $\Vb' X \Vb$ is singular.
			\item[(b)] $X$ is defective and zero is a defective eigenvalue.
		\end{enumerate} 
	\end{lemma}
	The proof is given in Appendix~\ref{app:proof_defective}. %
	From this lemma, we have that $\Vb' A_\lambda \Vb$ is singular if and only if zero is a defective eigenvalue of $A_\lambda$.   {In the context of the stability conditions to be presented in this paper,} the following assumption is considered.
	\begin{assumption}\label{aspt:regular}
		For $\lambda\in\Lambda$ associated with a given point $x_e$, zero is not a defective eigenvalue of $A_\lambda$, i.e., its algebraic and geometric multiplicities are the same.
	\end{assumption}
	
	{It is important to discuss whether this assumption is a restrictive condition.  Despite not holding for $\lambda$ such that $A_\lambda$ is a defective dynamic matrix {with zero as a defective eigenvalue}, e.g., a double-integrator,  Assumption~\ref{aspt:regular} is still fulfilled for many other systems of interest in the context of switched affine systems{, as it will be illustrated in our examples. Also, note that if one takes $A_\lambda$ randomly out of the set of singular matrices, the probability that it has repeated eigenvalues is {already} {zero}, which suggests that only a few particular cases are excluded by this assumption. However, a formal discussion in the sense of the measure of these sets is left for future work.}
	}
	
	{Finally, the main problem can be stated as: design a state-dependent switching rule $\sigma(t)=u(x(t))$ for system~\eqref{eq:sys} such that a chosen singular equilibrium point $x_e$, as given in Definition~\ref{def:singular_equilibrium}, is \textit{globally asymptotically stable}. In the following sections, we present sufficient conditions to achieve this goal.}
	
	\section{Stability conditions}\label{sec:stability}
	In this section, we present sufficient conditions to design a stabilizing state-dependent switching function $\sigma(t)=u\big(x(t)\big)$. Let us adopt the auxiliary state variable $\xi(t)= x(t)-x_e$, for which the error dynamics is
	\begin{equation}\label{eq:err_sy}
	\dot{\xi}(t) = A_{\sigma(t)}\xi(t) + \ell_{\sigma(t)}, \qquad\xi(0) = \xi_0:=x_0-x_e
	\end{equation}
	with $\ell_i = A_i x_e + b_i, ~\forall i\in \K$, and 
	a candidate Lyapunov function 
	\begin{equation}\label{eq:lyap}
	v(\xi) = \xi'\begin{bmatrix}
	\Vb&\Vp
	\end{bmatrix}\begin{bmatrix}
	\bar{P} & P_\times\\
	\bullet & P_{\perp}
	\end{bmatrix}\begin{bmatrix}
	\Vb&\Vp
	\end{bmatrix}'\xi.
	\end{equation} Decomposing the state variable as 
	\begin{equation}
	\xi =\Vb\bar{\xi} + \Vp \xi_\perp\label{eq:decomposition_xi}
	\end{equation}
	 we can also  define 
	\begin{align}
	f_{i}(\xi) &=\left.\frac{\partial v}{\partial \xi}'\dot{\xi}\right|_{\sigma=i} \nonumber\\
	&=\begin{bmatrix}
	\bar{\xi}\\\xi_{\perp} \\ 1
	\end{bmatrix}'\!\!\begin{bmatrix}
	\He\left(\bar{S}A_{i} \Vb \right) & \bullet & \bullet\\
	U_i' & \He\left(S_\perp A_{i} {V}_\perp \right)  &\bullet\\
	\ell_i'\bar{S}' & \ell_i'S_\perp'& 0
	\end{bmatrix}\!\!\begin{bmatrix}
	\bar{\xi}\\\xi_{\perp} \\ 1
	\end{bmatrix}\label{eq:v_dot}
	\end{align}
	with $\bar{S} = \bar{P}\Vb' +{{P}_\times} {\Vp'}$, $S_\perp  = {{P}_\perp} {\Vp'}+{P}_\times'\Vb'$ and $U_i = \Vb'A_{i}'S_\perp'  + 
	\bar{S}A_{i}{V}_\perp$, which follows from~\eqref{eq:identities}. 
	
	For a given $\lambda\in \Lambda$, which in our context is associated with the chosen singular equilibrium point $x_e$ 
	expressed as~\eqref{eq:xe_singular1}-\eqref{eq:xe_singular2}, %
	two different cases may occur:
	\begin{enumerate}
		\item The  nullspaces of $A_i,~\forall i\in\mathbb{K}$ include $\mathcal{N}$, so $\mathcal{N}$ is called a {\em shared nullspace}.
		\item The nullspace of at least one $A_i,~i\in\mathbb{K}$ does not include $\mathcal{N}$. In this case, $\mathcal{N}$ is a {\em particular nullspace}.
	\end{enumerate} 
	Note that for the first case, $\mathcal{N}$ is a common nullspace for the subsystems $A_i, ~i\in \mathbb{K}$ and, therefore, $A_i\Vp = 0, ~\forall i\in \K $. {This case is of great interest in this paper and is considered in the next theorem. Some remarks on the latter case will be provided subsequently.}
	
	\begin{theorem} \label{theo:common}Under the Assumption \ref{aspt:regular}, let a desired point $x_e$ satisfying~\eqref{eq:xe_singular1}-\eqref{eq:xe_singular2} and  an associated $\lambda\in\Lambda$ be given, such that the nullspace $\mathcal{N}$ of $A_\lambda$ is a shared nullspace for all $A_i,~i\in\mathbb{K}$. Define the matrix $M = \Vp'-\Vp'A_\lambda\Vb(\Vb' A_\lambda \Vb)^{-1 }\Vb'$ {where $\Vp\in\R^{n\times m}$ and $\Vb\in\R^{n\times p}$ have orthonormal columns spanning the nullspace and row space of $A_\lambda$,respectively}.
		{If $0\in{\rm Int} (\Pp),$ with $\Pp={\rm co} (\{M\ell_i~:~i\in\K\})\subset\R^m$}  and there exist matrices $\bar{P}$ and ${P}_\perp$  satisfying the LMIs
		\begin{equation}
		\He\left((\bar{P}\Vb' +{{P}_\times} {\Vp'})A_\lambda \Vb \right)<0\label{eq:lmi1}
		\end{equation}
		\begin{equation}\begin{bmatrix}
		\bar{P} & P_\times\\
		\bullet & P_{\perp}
		\end{bmatrix}>0\label{eq:lmi2}
		\end{equation}
		with ${P}_\times' = -P_\perp \Vp'A_\lambda\Vb(\Vb'A_\lambda\Vb)^{-1}$, then the state  $x(t)$ of~\eqref{eq:sys} globally asymptotically reaches the equilibrium point $x_e$  under the switching function
		\begin{equation}
		\sigma(t) = \arg\min_{i\in\K}  f_{i}(\xi(t)).\label{eq:switch_fun}
		\end{equation} 
	\end{theorem}
	\begin{proof}
		{Consider the candidate Lyapunov function $v(x)$ defined in~\eqref{eq:lyap}.} Let us demonstrate that, for system~\eqref{eq:sys} under the switching rule~\eqref{eq:switch_fun}, we have  $\dot{v}(\xi(t)) <0$ for all $t\geq 0$ such that $\xi(t)\neq 0$. First of all, we have from~\eqref{eq:lmi2} that $v(\xi)>0$ for all $\xi\neq 0$. From~\eqref{eq:v_dot} and~\eqref{eq:switch_fun}, we have
		\begin{align}
		\dot{v}(\xi(t)) &= \min_{i\in\K} f_{i}(\xi(t))\nonumber\\
		&=  \min_{\mu^* \in\Lambda} f_{\mu^*}(\xi(t))\nonumber\\
		&\leq f_\mu(\xi(t)),\quad \forall ~\mu\in\Lambda.\label{eq:upper_lim_v_dot}
		\end{align}
		In~\eqref{eq:upper_lim_v_dot}, the second equality and the inequality are due to properties of the min operator. {In view of the decomposition given in~\eqref{eq:decomposition_xi},} let us now consider two cases:  $\bar{\xi}\neq0$ and $\bar{\xi}=0$. 
		
		For $\bar{\xi}\neq0$, we evaluate~\eqref{eq:upper_lim_v_dot} with $\mu=\lambda$ to obtain
		\begin{align}
		\dot{v}(\xi(t)) &\leq \begin{bmatrix}
		\bar{\xi}\\\xi_{\perp} \\ 1
		\end{bmatrix}'\!\!\begin{bmatrix}
		\He\left(\bar{S}A_{\lambda} \Vb \right) & \bullet & \bullet\\
		U_\lambda' & \He\left(S_\perp A_{\lambda} {V}_\perp \right)  &\bullet\\
		\ell_\lambda'\bar{S}' & \ell_\lambda'S_\perp'& 0
		\end{bmatrix}\!\!\begin{bmatrix}
		\bar{\xi}\\\xi_{\perp} \\ 1
		\end{bmatrix}	\nonumber\\
		&= 2\bar \xi'\bar{S}A_{\lambda} \Vb \bar \xi\nonumber \\
		&< 0
		\label{eq:proof_xibar_nonzero}
		\end{align} 
		where the equality follows from the choice of $P_\times$, which makes $U_\lambda = 0$, from the fact that $A_\lambda \Vp = 0$ and that $\ell_\lambda =A_\lambda x_e+b_\lambda=0$  as a consequence of choosing $x_e$ satisfying~\eqref{eq:xe_singular1}-\eqref{eq:xe_singular2}. The last inequality is due to the LMI~\eqref{eq:lmi1}, recalling that $\bar{S} = \bar{P}\Vb' +{{P}_\times} {\Vp'}$. 
		
		Now, for the case where $\bar \xi= 0$, the function~\eqref{eq:v_dot} becomes
		\begin{equation}\label{eq:upper_bound_c2}
		f_i(\xi) = 2\xi_{\perp}'(S_\perp A_{i} {V}_\perp) \xi_{\perp}+2\xi_{\perp}'S_\perp \ell_{i} .
		\end{equation}
		Recalling that $A_i\Vp = 0, \forall i\in \mathbb{K}$ and noting that the chosen $P_\times$ makes $S_\perp=P_\perp M$, we can develop~\eqref{eq:upper_lim_v_dot} again to obtain 
		\begin{align}
		\dot{v}(\xi(t)) &=\min_{i\in\K} f_i(\xi(t))\nonumber\\
		&=  \min_{i\in\K} 2\xi_{\perp}'P_\perp M \ell_{i} \nonumber\\
		&=  \min_{\mu\in\Lambda} 2\xi_{\perp}'P_\perp M \ell_{\mu} \nonumber\\
		&\leq  -\epsilon 2\xi_{\perp}'P_\perp \xi_{\perp}  \nonumber\\
		&<0\label{eq:upper_lim_v_dot2}
		\end{align}
		where the first inequality is due to $-\epsilon\xi_{\perp}\in {\rm co} (\{M\ell_i:i\in\K\})$ for some $\epsilon>0$ small enough, which is implied by the condition {$0\in{\rm Int} (\Pp)$, where we recall that $\Pp ={\rm co} (\{M\ell_i:i\in\K\})$}.
		Finally, for the two studied cases, we have that $\dot{v}(\xi(t))<0, ~\forall \xi(t)\neq 0$ and, given the continuity of $\dot{v}(\xi)$ with respect to $\xi$, we can conclude that  $x_e$ is globally asymptotically stable {from Theorem~1 in~\cite[p. 153]{filippov1988differential}}.
	\end{proof}
	This is our first result dealing with global asymptotic stability of a singular equilibrium point $x_e\in X_e$ as presented in Definition~\ref{def:singular_equilibrium}. For this case, besides the feasibility of the LMIs~\eqref{eq:lmi1}-\eqref{eq:lmi2}, there exists a condition on the {polytope} ${\Pp \subset \R^m}$, which must contain the origin in its interior. 
	
	{
		\begin{remark}\label{rem:zeroinP}The condition $0\in{\rm Int} (\Pp)$ with $\mathbb{P}\subset \R^m$ can be verified, for instance, by checking if $\rank([M\ell_i]_{i\in\K})=m$ and if there exists  $\mu\in{\rm Int}(\Lambda)$ such that $M\ell_\mu=0$. 
	Note that these conditions can be easily checked and are not only sufficient for but also equivalent to $0\in{\rm Int} (\Pp)$. To verify the necessity assume $0\in{\rm Int} (\Pp)$, which implies that $\rank([M\ell_i]_{i\in\K})=m$ and that for an arbitrary $\mu^\ast\in{\rm Int}(\Lambda)$, we have $-\epsilon M\ell_{\mu^\ast}\in\Pp$ for some $\epsilon>0$. Define the vector $\mu^-\in\Lambda$ such that $M\ell_{\mu^-}=-\epsilon M\ell_{\mu^\ast}$. Then, the vector formed by the weighted average $\mu = (\epsilon+1)^{-1}(\epsilon\mu^\ast+\mu^-)\in{\rm Int}(\Lambda)$ verifies $M\ell_\mu=0$. 
	Notice that one particular case may arise when the vector $\lambda$ associated with $x_e$ belongs to ${\rm Int}(\Lambda)$. In this case, the general condition $0\in{\rm Int} (\Pp)$ becomes simply $\rank([M\ell_i]_{i\in\K})=m$.
\end{remark}} 
	
	As one can notice from the proof of Theorem~\ref{theo:common}, an interpretation for the requirement that the origin must lie in the interior of the convex combination of $M\ell_i,~i\in\K$, is the following: to guide $x_\perp$ through the nullspace of dimension $m$ towards $x_{e\perp}$, at least $m$ linearly independent vectors $M\ell_i$ capable of driving the state to all directions in the equilibrium subspace are necessary. For instance, in the case of a switched linear system $\dot{x}(t)=A_{\sigma(t)} x(t)$ with all matrices $A_i$ sharing a nullspace $\mathcal{N}$, the component $x_\perp(t)=\Vp' x(t)$ remains constant for all $t\geq0$ and for any switching signal $\sigma(t)$ {if $x(0)\in\mathcal{N}$}. Thus, the {presence} of affine terms is crucial to ensure global asymptotic stability in switched systems with shared nullspaces. The following example borrowed from~\cite{scharlau2014switching} demonstrates how this condition can be used to stabilize a switched affine system with a shared nullspace of $\dim(\mathcal{N})=1$.
	
	\begin{example}[{Example 2 of \cite{scharlau2014switching}}]\label{ex:scharlau}
		Let us consider the system defined in Example 2 of~\cite{scharlau2014switching}, given as~\eqref{eq:sys} with
		\begin{equation}
		A_1=A_2=\begin{bmatrix}
		0 & 0\\
		0 & 0
		\end{bmatrix},~A_3=\begin{bmatrix}
		0 & 0\\
		0 & -1
		\end{bmatrix},~ b_1=\begin{bmatrix}
		-1\\0
		\end{bmatrix}
		\end{equation}
		$b_2=-b_1$ and $b_3=0$. Any $\lambda\in\Lambda$ such that $\lambda_1=\lambda_2$ is associated with the singular equilibrium point $x_e=0~{\in\R^2}$ for this system, since the matrices $A_i,~i\in\K$ share the nullspace $\mathcal{N}$ spanned by $V_\perp = [1~~0]'$. {Consider $\bar{V}=[0~~1]'$.} Notice that any point $x_e=[\tau~~0]',~\tau\in\R$, could also be chosen. Let us choose, for instance, $\lambda=[1/3~~1/3~~1/3]'$ allowing us to verify that scalars $\bar{P}=1.5$ and $P_\perp=1$ satisfy the LMIs~\eqref{eq:lmi1}-\eqref{eq:lmi2}. Also, notice that the matrix {$[M\ell_i]_{i\in\K}=[-1 ~ 1 ~  0]$ }has rank 1 {and that $\lambda\in{\rm Int}(\Lambda)$}, which implies that {$0\in{\rm Int} (\Pp)$} as required in the statement of Theorem~\ref{theo:common}. Then, one can define the Lyapunov function~\eqref{eq:lyap} to implement the state-dependent switching rule~\eqref{eq:switch_fun} and ensure global asymptotic stability of $x_e=0$. Figure~\ref{fig:scharlau} shows the state trajectories evolving from $x_0 =[-4~~~5]'$ and converging to $x_e$, the obtained switching signal and the corresponding value of the Lyapunov function~\eqref{eq:lyap} over $x(t)$.
		\begin{figure}\centering
			\def\svgwidth{.5\linewidth}
			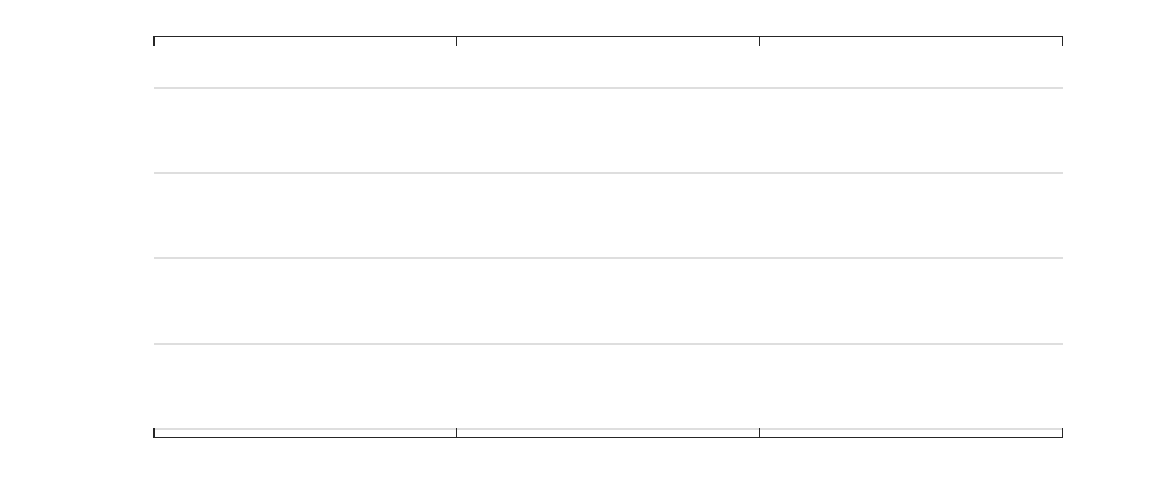\\
			\def\svgwidth{.5\linewidth}
			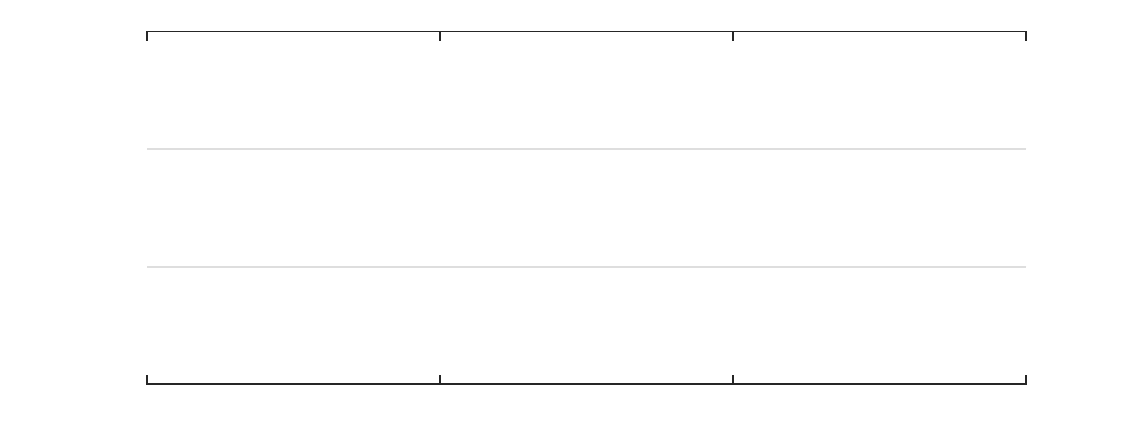\\
			\def\svgwidth{.5\linewidth}
			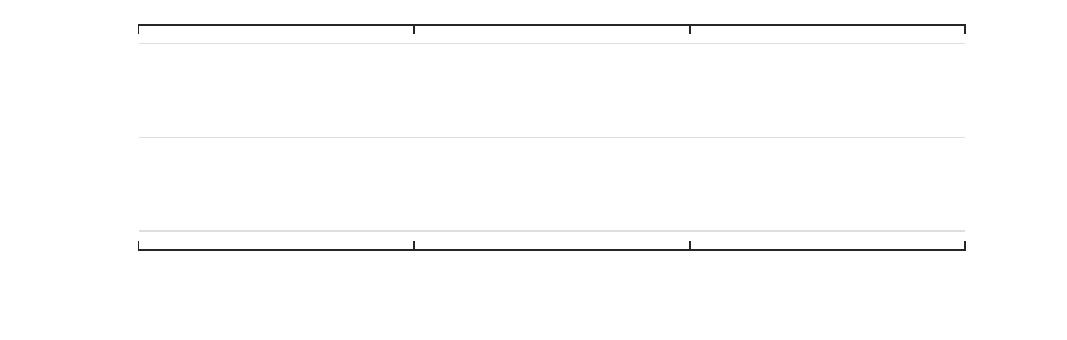
			\caption{State trajectories (top) $x_1(t)$  and $x_2(t)$ (green and blue, respectively) evolving from $x_0 =[-4~~5]'$, corresponding  Lyapunov function $v(x(t))$ (middle) and  switching signal $\sigma(t)$ generated by the switching function~\eqref{eq:switch_fun} (bottom).}\label{fig:scharlau}
		\end{figure}
	\end{example}
	{Applying Theorem 1 from~\cite{scharlau2014switching} in Example \ref{ex:scharlau}, the stabilization conditions were not able to ensure asymptotic stability directly, since the inequality that should be {strictly} negative to impose the decreasing of the Lyapunov function for all $t\geq 0$ was only bounded by a negative semi-definite function. The authors, however, presented an \textit{ad hoc} argument to demonstrate that for this particular example {their} designed switching function still guarantees global asymptotic convergence to the equilibrium. Notice that the stability conditions of Theorem \ref{theo:common} in the present paper handle this case {directly,} without requiring this \textit{a posteriori} analysis.}
	
	Another important remark {about} Theorem 1 is that it handles the cases studied in~\cite{beneux2018integral} and~\cite{ndoye2019robust} where a switched affine system~\eqref{eq:sys} that presents a Hurwitz convex combination $A_\lambda$, for a given $\lambda\in\Lambda$, {is augmented by adding new state variables that introduce integrators to the system dynamics.} The resulting augmented system improves the steady-state performance due to the integral action. Indeed, consider the augmented error system, containing as a particular case the dynamics studied in~\cite{beneux2018integral,ndoye2019robust},
	\begin{equation}
	\dot{\tilde{\xi}}(t) = \tilde{A}_{\sigma(t)}\tilde{\xi}(t) + \tilde{\ell}_{\sigma(t)}\label{eq:aug_sys}
	\end{equation}
	with $\tilde{\xi}(t) = [\xi(t)'~~z(t)']'$, $z(t) = \int_0^t C_{\sigma(\tau)}\xi(\tau) d\tau$ and
	\begin{equation}
	\tilde{A}_{i} = \begin{bmatrix}
	A_i &0\\
	C_i &0
	\end{bmatrix},~~\tilde{\ell_i} =\begin{bmatrix}
	\ell_i\\
	0
	\end{bmatrix}
	\end{equation}
	where $C_i \in\R^{m\times n},~i\in\K,$ are chosen by the designer. Evaluating the conditions of Theorem~\ref{theo:common} for the augmented system in~\eqref{eq:aug_sys}, we can conclude that they hold if $A_\lambda$ is Hurwitz and $0\in{\rm Int} \big({\rm co} (C_\lambda A_\lambda^{-1}\ell_i ~:~{i\in\K})\big)$.
	
	{Finally, Theorem \ref{theo:common} allows a generalization for the case where the nullspace $\mathcal{N}$ is shared by some of the matrices $A_i,~i\in\K$ but not all. This is given in the following corollary.
	\begin{corollary}\label{coro:shared}
		Consider a switched affine system~\eqref{eq:sys} and a given $\lambda\in\Lambda$ for which the nullspace of $A_\lambda$ is not shared by all $A_i,~i\in\K$  but only by $A_i,~i\in\K_\perp\subset\K$.		For this system, the statement of Theorem \ref{theo:common} remains valid whenever the condition {$0\in{\rm Int} (\Pp)$ is replaced by $0\in{\rm Int} (\Pp_\perp)$ with $\Pp_\perp={\rm co} (\{M\ell_i:i\in\K_\perp\})$}.
	\end{corollary}
	\begin{proof}
		The proof is identical to the one of Theorem~\ref{theo:common}, but considering $\K_\perp$ instead of $\K$ in~\eqref{eq:upper_lim_v_dot2}.
	\end{proof}
	{\begin{remark}
		The general case where the nullspace $\mathcal{N}$ is a \emph{particular nullspace}, as defined before Theorem~\ref{theo:common}, and $0\notin{\rm Int} (\Pp_\perp)$ is more intricate because $A_i\Vp\neq0$ for some $i\in\K$, which precludes the developments in~\eqref{eq:upper_lim_v_dot2}. However, analogous conditions may be developed when $0\in{\rm Int} \big({\rm co} (\{M(A_i\Vp z+\ell_i):i\in\K\})\big)$ for all $z\in\R^m$. This is equivalent to verifying if the origin is contained in the interior of the convex hull of a set of affine vector functions. This more general and involved problem will be tackled in future work.
	\end{remark}}

	\section{Convergence rate analysis}\label{sec:conv_rate}
	
	The convergence rate to singular equilibrium points is not globally exponential, in the general case. This has also been observed for the particular case studied in~\cite{beneux2018integral}. {For instance, a subclass of systems for which exponential stability cannot be globally ensured is given in Appendix~\ref{app:proof_violated}.}

	{This fact is also }reflected on the results of Theorem~\ref{theo:common} in the sense that one may not obtain a quadratic negative-definite upper bound on the time derivative of the Lyapunov function that holds globally. This motivates the investigation of what is the {local} convergence rate ensured by the design procedure presented in Theorem~\ref{theo:common} inside some given set.  To answer this question we investigated the local guaranteed convergence rate in the next corollary. Before presenting it, let us define a vector function $\gamma:{\Lambda}\times\R\times\K\rightarrow\Lambda$ as 
	\begin{equation}
	\gamma_j({\lambda},\epsilon,i) = \left\{ \begin{array}{cl}
	\lambda_j + \epsilon(1-\lambda_j),& ~i=j\\
	\lambda_j - \epsilon\lambda_j,& ~i\neq j
	\end{array}\right. \label{eq:lambda_eps}
	\end{equation}
	Notice that, for each $i\in\K$ {and $\lambda\in\Lambda$}, this function is only defined for $\epsilon\in[\underline{\epsilon}_i,1]$ with $\underline{\epsilon}_{i} = \max\big( \{\lambda_i/(\lambda_i-1)\} \cup \{(\lambda_j-1)/\lambda_j:{j\in \K}, ~j\neq i\}\big)$, which ensures $\gamma(\lambda,\epsilon,i)\in\Lambda$. Also, notice that $X_{\gamma(\lambda,\epsilon,i)}= (1-\epsilon)X_\lambda +\epsilon X_i$ for arbitrary matrices $X_j,~j\in\K$. Finally, whenever the strict LMI~\eqref{eq:lmi1} is satisfied, note that there always exist an $\epsilon>0$ and a matrix $G_\epsilon>0$ such that the inequalities 
	\begin{align}
	\He\left((\bar{P}\Vb' +{{P}_\times} {\Vp'})A_{\gamma(\lambda,\epsilon,i)} \Vb \right)&<-G_\epsilon\label{eq:gamma_eps_plus}\\
	\He\left((\bar{P}\Vb' +{{P}_\times} {\Vp'})A_{\gamma(\lambda,-\epsilon,i)} \Vb \right)&<-G_\epsilon\label{eq:gamma_eps_minus}\
	\end{align}
	are feasible for all $i\in \K_a({\lambda})$ with $\K_a(\cdot)$ defined as \begin{equation}\label{eq:Ka}
	\K_a(\mu) = \{i\in \K:\mu_i\neq 0, ~\mu\in \Lambda\}
	\end{equation}
	\begin{corollary}\label{coro:coro01}
		Let $L=[\ell_i]_{i\in\K_a(\lambda)}$. If  $\rank(ML)=m$, the switching function~\eqref{eq:switch_fun} designed through the conditions in Theorem~\ref{theo:common} locally ensures the exponential convergence 
		\begin{equation}\label{eq:exp_convergence}
		v(\xi) \leq e^{-\alpha t} v(\xi(0))
		\end{equation}
		with convergence rate $\alpha = s_{\min}(Q)/s_{\max}(P)$ where 
		\begin{equation}\label{eq:Q_and_P}
		Q= \begin{bmatrix}
		G_\epsilon&0\\
		0&0
		\end{bmatrix}\!+\! \frac{\epsilon\beta}{|\K_a(\lambda)|}\!\begin{bmatrix}
		\bar{S}L\\S_\perp L
		\end{bmatrix}\!\! \begin{bmatrix}
		\bar{S}L\\S_\perp L
		\end{bmatrix}'\!\!, ~~P = \begin{bmatrix}
		\bar{P} & P_\times\\
		\bullet & P_{\perp}
		\end{bmatrix}
		\end{equation}
		with a given $\beta>0$, a matrix $G_\epsilon>0$ and a scalar $\epsilon>0$ satisfying~\eqref{eq:gamma_eps_plus}-\eqref{eq:gamma_eps_minus} for all $i\in \K_a(\lambda)$. This exponential convergence rate holds within the level set
		\begin{equation}
		\Omega{(r)}=\left\{\xi \in \R^n~: ~v(\xi)\leq r \right\}\label{eq:level_set}
		\end{equation}
		for any $r>0$ such that $\Omega{(r)}\subseteq\mathcal{X}$ with
		\begin{equation}
		\mathcal{X} = \bigcap_{i\in\K_a(\lambda)} \tilde\Omega_i \cap \{\xi\in\R^n~:~ |g_i(\xi)|\leq1/\beta\}\label{eq:region_omega}
		\end{equation}
		\begin{equation}\label{eq:omega_tilde}
		\tilde \Omega_i=\{\xi\in\R^{n}:|g_i(\xi)| \geq\beta h_i(\xi) \}
		\end{equation}
		and the functions $g_i(\cdot) $ and  $h_i(\cdot)$ are given by
		\begin{equation}\label{eq:g_i}
		g_i(\xi) = \bar{\xi}'U_i\xi_\perp+ \bar{\xi}'\bar{S}\ell_i + \xi_\perp' S_\perp\ell_i
		\end{equation}
		\begin{equation}\label{eq:h_i}
		h_i(\xi):=\frac{1}{2}\begin{bmatrix}
		\bar{\xi} \\ \xi_\perp
		\end{bmatrix}'\!\! \begin{bmatrix}
		\bar{S}\ell_i\\S_\perp \ell_i
		\end{bmatrix}\!\! \begin{bmatrix}
		\bar{S}\ell_i\\S_\perp \ell_i
		\end{bmatrix}'\!\!\begin{bmatrix}
		\bar{\xi} \\ \xi_\perp
		\end{bmatrix} .
		\end{equation}
	\end{corollary}
	The proof is given in Appendix \ref{app:B}.
	
	A few observations are pertinent at this point. The first is that the only additional requirement in Corollary~\ref{coro:coro01} with respect to Theorem~\ref{theo:common} is $\rank(ML)=m$. Also, notice that for an index independent matrix $A_i=A,~\forall i\in\K$, the matrix $G_\epsilon$ is independent of $\epsilon$ and this scalar can be chosen as large as possible if $\gamma_i(\epsilon,i),\gamma_i(-\epsilon,i)\in\Lambda$ for all $i\in\K_a(\lambda)$. This clearly is the best choice for $\epsilon$ in terms of obtaining the tightest estimate for the convergence rate $\alpha$ in~\eqref{eq:exp_convergence}. Note that $\alpha$ is non-decreasing as $\epsilon$ and $\beta$ increases, in this case. Another remark is that $\beta$ states a trade-off between  the convergence rate estimate and the size of the region $\Omega{(r)}$ on which this estimate holds. This readily follows from the definition of this region in~\eqref{eq:region_omega} that takes into account the inequality $|g_i(\xi)|\leq1/\beta$, indicating that this region shrinks as $\beta\rightarrow \infty$.

	Let us present a polynomial-time numerical procedure to estimate the convergence rate for a given level set $\Omega{(r)}$ of the Lyapunov function. For given matrices $\bar{P}$ and $P_\perp$ defining the Lyapunov function~\eqref{eq:lyap} and a level $r>0$, one possible approach to obtain this estimate is described in the following two steps:
	\begin{enumerate}
		\item Find a suitable $\beta>0$ for the given level $r>0$, by solving the sum-of-squares (SOS) problem 
		\begin{equation}
		\max_{\phi_i(\xi),\psi_i(\xi)\in\Sigma^2(\xi),~\beta>0} \beta \qquad\mathrm{s.t.}
		\end{equation}
		\begin{equation}	\label{eq:sos_1}
		g_i(\xi)^2\!-\beta^{2}\!h_i(\xi)^2-\phi_i(\xi)\big(r-v(\xi)\big)  \in \Sigma^2(\xi),
		\end{equation}	
		\begin{equation}	\label{eq:sos_2}
		1\!-\!\beta^{2}g_i(\xi)^2-\psi_i(\xi)\big(r-v(\xi)\big)  \in \Sigma^2(\xi),
		\end{equation}
		for all $i\in\K_a(\lambda)$ where $\Sigma^2(\xi)$ is the set of SOS polynomials of $\xi\in\R^n$ and with $h_i(\xi)$ and $g_i(\xi)$ defined in the proof of Corollary~\ref{coro:coro01}. {Recalling the Putinar's Positivstellensatz \cite{putinar1993positive}, we can } verify that when $\xi\in\R^n$ is such that $v(\xi)\leq r$, ~\eqref{eq:sos_1} implies that $|g_i(\xi)|\geq\beta h_i(\xi)$  and~\eqref{eq:sos_2} implies that $|g_i(\xi)|\leq 1/\beta $.
		\item Search for $\epsilon\in[0,1]$ that provides the largest convergence rate $\alpha=s_{\min} (Q)/s_{\max} (P)$, where $Q$ and $P$ are given in~\eqref{eq:Q_and_P} and $\gamma(\lambda,\epsilon,i),\gamma(\lambda,-\epsilon,i)\in\Lambda$ for all $i\in\K_a(\lambda)$. This can be done by finding $G_\epsilon>0$ for each $\epsilon$ such that the inequalities~\eqref{eq:gamma_eps_plus}-\eqref{eq:gamma_eps_minus} hold for all $i\in\K_a(\lambda)$ and using it to evaluate $s_{\min} (Q)$, which can be efficiently done by semidefinite programming. Alternatively, imposing $G_\epsilon = \rho I>0$ allows us to obtain $-\rho$ directly from the maximum eigenvalue of the matrices on the left-hand side of~\eqref{eq:gamma_eps_plus}-\eqref{eq:gamma_eps_minus} for all $i\in\K_a(\lambda)$.
	\end{enumerate}

	The following example illustrates this procedure.
	\begin{example}
		{Consider the switched affine system~\eqref{eq:sys} given by $N=3$ marginally stable subsystems
		\begin{equation}
		A_1\!=\!\!
		\begin{bmatrix}
		\!-6 & 5 & 0 \\ 
		2 & \!-7 & 0 \\ 
		\!-2 & 0 & 0
		\end{bmatrix} \!\!,A_2\!=\!\!\begin{bmatrix}
		\!-6 & 2 & 0 \\ 
		2 & \!-7 & 0 \\ 
		2 & 3 & 0
		\end{bmatrix}\!\!,A_3\!=\!\!\begin{bmatrix}
		-\!3 & -\!1 & 0 \\ 
		-\!1 & -\!1 & 0 \\ 
		2 & -\!3 & 0
		\end{bmatrix} 
		\end{equation}
		\begin{equation}
		b_1=\begin{bmatrix}
		1 \\ 
		-1 \\ 
		0
		\end{bmatrix} \!,~b_2=\begin{bmatrix}
		-1 \\ 
		1 \\ 
		2
		\end{bmatrix}  \!,~b_3=\begin{bmatrix}
		0 \\ 
		0 \\ 
		-2
		\end{bmatrix}\!.
		\end{equation}
		For $\lambda=[\frac{1}{3}~\frac{1}{3}~\frac{1}{3}]'\in\Lambda$ we can verify that $x_e=0\in\R^3$ is a singular equilibrium point for the system and the design conditions in Theorem~\ref{theo:common} are satisfied for   $[\Vb ~~\Vp] = I$ and
		\begin{equation}
		\bar{P}\!=\!\begin{bmatrix}
		1.1989 & \!-0.0046 \\ 
		\!-0.0046 & 1.2087
		\end{bmatrix}\!\times \!10^{-3}\!,~P_\perp\!=\!
		1.1542 \times 10^{-3}
		\end{equation}
		which were obtained by minimizing the condition number of $P$ subject to $P>10^{-3} I$. As $\rank(ML)=1$, the requirements in Corollary~\ref{coro:coro01} are satisfied and the SOS-based procedure described in this section allows us to obtain a guaranteed convergence rate $\alpha$ such that~\eqref{eq:exp_convergence} holds within a sublevel set~\eqref{eq:level_set} of the Lyapunov function~\eqref{eq:lyap}. Searching for polynomials $\phi_i,\psi_i\in\Sigma^2(\xi)$ of degree 2, Figure~\ref{fig:conv_rate} shows the guaranteed rate $\alpha$ (in blue) obtained from our method with respect to the euclidean ball $\B(0,R)=\{\xi\in\R^{n}~:~\xi'\xi\leq R\}$ for several values of $R>0$, where the smallest level set of the Lyapunov function $\Omega(r)$ containing $\B(0,R)$ is given by $r=Rs_{\max}(P)$. For comparison, the largest $\alpha$ ensured by the LMI-procedure given in~\cite[Proposition~2]{hetel2014local} is depicted (in red). Note that, for this example,~\cite{hetel2014local} guarantees a faster convergence (i.e., larger values of $\alpha$) when $R$ is small but, for $R\geq1$ Corollary~\ref{coro:coro01} yields better values of $\alpha$. Also, for $R$ larger than $2.512$ in the given grid of points, feasible solutions could not be found under the procedure given in~\cite[Proposition~2]{hetel2014local}.
		\begin{figure}
			\centering
			\def\svgwidth{.5\linewidth}
			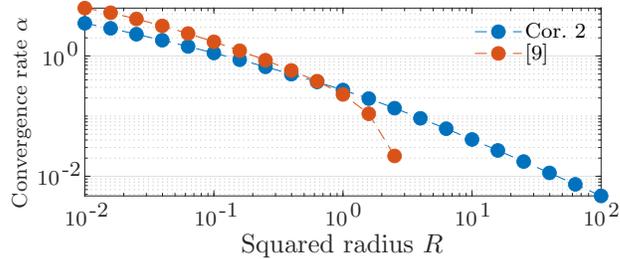
			\caption{Guaranteed convergence rate $\alpha$ obtained from the method in Corollary~\ref{coro:coro01} and by~\cite[Proposition~2]{hetel2014local} valid inside sets $\B(0,R)=\{\xi\in\R^{n}~:~\xi'\xi\leq R\}$ of squared radius $R>0$.}\label{fig:conv_rate}
		\end{figure}}
	\end{example}

	To conclude this section, we present some technical remarks that allow us to improve the estimate provided by this method.  First, testing for SOS polynomials {of a fixed degree} is known to be more conservative than testing for non-negativeness. However, polynomial-time algorithms for SOS programming are appealing in our context since optimizing over polynomials with non-negativeness as a constraint is also known to be NP-hard. Second, the average of the upper bounds calculated in~\eqref{eq:v_quad_local_upperbound} in the proof of Corollary~\ref{coro:coro01} could be any weighted average and better estimates for the convergence rate may be obtained by properly adjusting these weights, {which can be also done by semidefinite programming}. Finally, an index-dependent $\epsilon_i$ in~\eqref{eq:gamma_eps_plus}-\eqref{eq:gamma_eps_minus} can equally be adopted instead of an index-independent $\epsilon$.

	\section{Practical Example}
	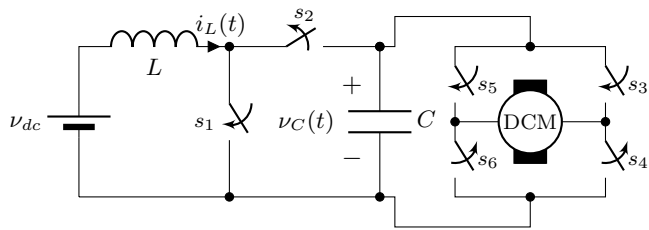
\begin{figure}[b]\centering\scalebox{1}{
			\begin{tikzpicture}[scale=2,american]
			\draw (0,0) to[battery2,invert,l={\footnotesize $\nu_{dc}$}] (0,1) to[L,-*,l_={\footnotesize$L$},i={\footnotesize$~~i_L(t)$}] (1,1) to[cspst,-*,l_={\scriptsize $s_1$}] (1,0) -- (0,0);
			\draw (1,1) to[ospst,-*,l={\scriptsize$s_2$}] (2,1) to[C,*-*,l={\footnotesize$C$},v={\footnotesize$\nu_C(t)\!\!\!$}] (2,0) -- (1,0);
			\draw (2,1) -| (2.1,1.2) -| (3,1) to[open,*-*] (3,0) |- (2.1,-.2) |- (2,0); 
			\draw (3,1) -- (3.5,1) to[cspst,-*,l={\scriptsize $s_3$}] (3.5,.5) to[ospst,l={\scriptsize $s_4$}] (3.5,0) --(3,0);
			\draw (3,1) -- (2.5,1) to[cspst,-*,l={\scriptsize $s_5$}] (2.5,.5) to[ospst,l={\scriptsize $s_6$}] (2.5,0) --(3,0);
			\draw (2.5,0.5) --(3,0.5) node[elmech]{\scriptsize DCM} --(3.5,0.5);
			\end{tikzpicture}}
		\caption{Circuit schematic of a boost dc-dc converter and an h-bridge feeding a dc motor.}\label{fig:dc_motor}
	\end{figure}
	In this {section}, we illustrate two of many possible applications of the presented theory. Consider the circuit represented in Figure~\ref{fig:dc_motor}, which schematizes a dc motor controlled by an h-bridge, characterized by the switches $(s_3,s_4,s_5,s_6)$. To step-up the input voltage $\nu_{dc}$, a boost converter with switches $(s_1,s_2)$ is implemented and connected to the h-bridge input. We investigated two problems of interest for this system, namely, the position control and velocity control with integral action.
	\subsection{Rotational position control}
	Neglecting the dc motor inductance and considering that each pair of switches $(s_1,s_2)$, $(s_3,s_4)$ and $(s_5,s_6)$ are alternately commanded, the dynamic model of this system is given as a switched affine system~\eqref{eq:sys} with matrices
	\begin{equation}
	A_i \!=\! \begin{bmatrix}
	-\frac{R_L}{L} 	& -\frac{u_1(i)}{L}			     		& 0   									& 0\\
	\frac{u_1(i)}{C}& -\frac{u_3(i)}{R_mC} 					& \frac{(2u_2(i)-1)u_3(i)K_e}{R_mC} 	& 0\\
	0	   			& \frac{(2u_2(i)-1)u_3(i)K_e}{JR_m}		& -\frac{K_e^2+cR_m}{JR_m}				& 0\\
	0				&	0									& 1										& 0 	
	\end{bmatrix}\label{eq:A_i_motor}
	\end{equation}
	and $b_i=[\nu_{dc}/L~ 0~ 0~ 0]'$ for all $i\in\K:=\{1,\dots,8\}$. The boolean variables $u_1(i)$, $u_2(i)$ and $u_3(i)$ take values according to Table~\ref{tab:bool_dcmotor} for given configurations of the switches $(s_1,\dots,s_6)$. The state vector is constructed as $x(t)=[i_L(t)~~\nu_C(t)~~\omega(t)~~\theta(t)]'$ where $i_L(t)$ and $\nu_C(t)$ are indicated in Figure~\ref{fig:dc_motor}, whereas $\omega(t)$ and $\theta(t)$ are the rotor shaft angular velocity and position, respectively. The remaining parameters have their values given in Table~\ref{tab:par_dcmotor}.
	\begin{table}\centering
		\caption{Values of the boolean variables $u_1(\sigma)$, $u_2(\sigma)$ and $u_3(\sigma)$ and switches $s_1$, $s_2$ and $s_3$ with respect to the active subsystem $\sigma\in\K$ ($s_1=1$ when $s_1$ is closed, and \textit{vice-versa}). }\label{tab:bool_dcmotor}
		\begin{tabular}{c||c|c|c|c|c|c|c|c}
			& \multicolumn{8}{|c}{$\sigma$}\\
			& 1 & 2 & 3 & 4 & 5 & 6 & 7 & 8\\\hline\hline
			$u_1(\sigma)$	& 0 & 1 & 0 & 1 & 0 & 1 & 0 & 1\\
			$u_2(\sigma)$	& 0 & 0 & 1 & 1 & 0 & 0 & 1 & 1\\
			$u_3(\sigma)$	& 0 & 0 & 0 & 0 & 1 & 1 & 1 & 1\\\hline
			$s_1$           & 1 & 0 & 1 & 0 & 1 & 0 & 1 & 0\\
			$s_3$           & 0 & 0 & 1 & 1 & 0 & 0 & 1 & 1\\
			$s_5$           & 0 & 0 & 1 & 1 & 1 & 1 & 0 & 0\\
		\end{tabular}
	\end{table}
	\begin{table}
		\centering
		\caption{Numerical values for the system parameters.}\label{tab:par_dcmotor}
		\scalebox{0.9}{\begin{tabular}{crc}
				Symbol 		& Quantity 					  			& Value 							\\\hline
				$R_L$  		& Resistance of the inductor $L$ 		& 0.5 $\Omega$ 						\\
				$L$    		& Inductance of $L$				  		& 1 $\times 10^{-3}$ H		   		\\
				$C$    		& Capacitance of the capacitor $C$		& 2 $\times 10^{-3}$ F      		\\
				$K_e$  		& Electric constant of the motor		& 5 $\times 10^{-3}$ V.s/rad		\\
				$R_m$		& Resistance of the motor winding		& 1 $\Omega$						\\
				$J$	   		& Inertia of the motor shaft			& 1 $\times 10^{-6}$  kg.m$^2$		\\
				$c$	   		& Viscous friction on the motor shaft	& 1 $\times 10^{-4}$  kg.s.m$^2$	\\
				$\nu_{dc}$	& Input voltage							& 12 V
		\end{tabular}}
	\end{table}
	
	This system has a shared nullspace $\mathcal{N}$ of $\dim(\mathcal{N})=m=1$ spanned by $V_\perp=[0~0~0~1]'$, which corresponds to the rotational position $\theta(t)$. {Consider $[\bar{V}~ V_\perp] = I$.  Notice that for all $\lambda\in\Lambda$, we have $\rank(A_{\lambda})=m<n=4$, as the forth column of $A_\lambda$ is always null. For the goal point $x_e = [0~~\nu_{Ce}~~0~~\theta_e]'$ with $\nu_{Ce}=24$~V and an arbitrary $\theta_e\in\R$}, we can verify that the conditions in Definition~\ref{def:singular_equilibrium}  stating that $x_e$ is a singular equilibrium point are fulfilled by $\lambda =[\frac{1}{4}~\frac{1}{4}~\frac{1}{4}~\frac{1}{4}~0~0~0~0]'\in\Lambda$. Recalling the result from Theorem~\ref{theo:common}, notice that the matrices \begin{equation}
	\bar{P}\!=\!\begin{bmatrix}
	1.4953 & 1.1691 & 0.0000 \\ 
	1.1691 & 3.7599 & 0.0000 \\ 
	0.0000 & 0.0000 & 1.2560
	\end{bmatrix}	\!\!\times\!\!10^{-3}\!\!, ~~P_\perp\!=\!2.0007,\nonumber
	\end{equation}
	satisfy the LMIs~\eqref{eq:lmi1}-\eqref{eq:lmi2}. Additionally, although $\lambda\notin{\rm Int}(\Lambda)$ in this case, we still have {$0\in{\rm Int} (\Pp)$ }as {$\mu =[\frac{1}{8}~\frac{1}{8}~\frac{1}{8}~\frac{1}{8}~\frac{1}{8}~\frac{1}{8}~\frac{1}{8}~\frac{1}{8}]'\in{\rm Int}(\Lambda)$ fulfills $M\ell_{\mu}=0$ and $\rank([M\ell_i]_{i\in\K})=m$, as discussed in Remark~\ref{rem:zeroinP}.} Then, we can consider the Lyapunov function~\eqref{eq:lyap}, which allows the implementation of the state-dependent switching function~\eqref{eq:switch_fun}, making $x_e$ globally asymptotically stable. Varying $\theta_e\in\{\pi,2\pi,-\pi, 0\}$ every 1~s, we obtained the state trajectories depicted in Figure~\ref{fig:state_motor}, evolving from null initial conditions. {Notice that, the fact that $x_e$ is in the nullspace of $A_\lambda$ allows us to perform an online update of $\theta_e\in\R$ while keeping the other components of $x_e$ constant, which characterizes the tracking control of a piecewise constant trajectory.} This simulation shows how the proposed switching function can successfully perform the position control of a dc motor fed by a switching circuit as {given} in Figure~\ref{fig:dc_motor}. In the following subsection another application of this theory is presented.
	
	\begin{figure}
		\centering
		\def\svgwidth{.55\linewidth}
		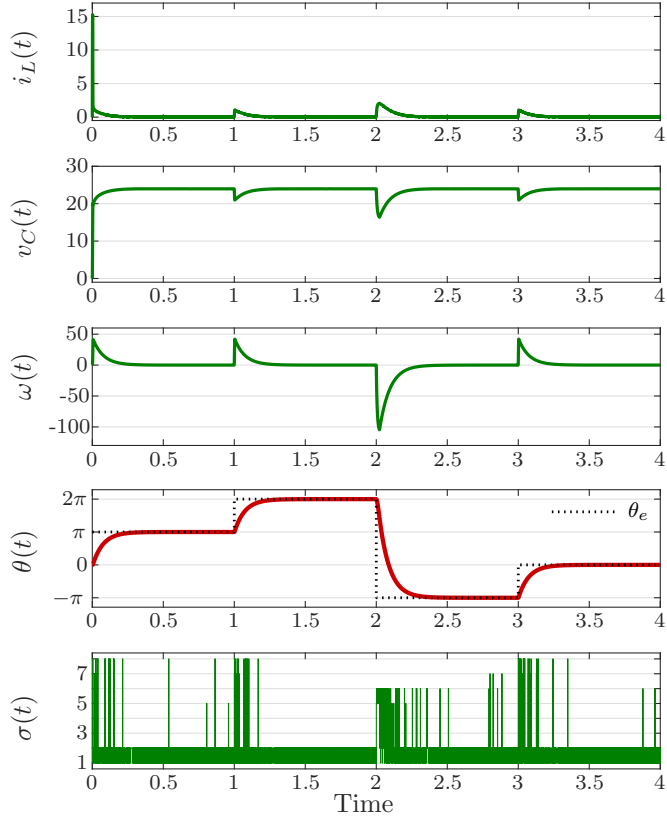
		\caption{Position control: State trajectories $i_L(t),~\nu_C(t),~\omega(t)$ and $\theta(t)$ of the converter and the dc motor and switching signal $\sigma(t)$ generated by the state-dependent switching function~\eqref{eq:switch_fun}. The position reference $\theta_e$ varies inside the set $\{\pi,2\pi,-\pi, 0\}$.}\label{fig:state_motor}
	\end{figure}

	\subsection{Velocity control with integral action}
	The proposed switching function is also able to control the rotational velocity of a dc motor with integral action, as discussed in Section~\ref{sec:stability}. To empirically investigate the robustness of our switching function, we consider the system given in Figure~\ref{fig:dc_motor} under an unknown external torque $\tau(t)$ actuating on the motor shaft. Considering a state variable  $x(t)= [i_L(t)~~\nu_C(t)~~\omega(t)~~\int_{0}^{t}(\omega(t)-\omega_e)dt]'$, the system with the integrator is given by matrices $A_i$ defined in~\eqref{eq:A_i_motor} and $b_i=[\nu_{dc}/L~ 0~ 0~ -\!\omega_e]',~i\in\K$. The vector $\lambda=[0~0~0~0~0~0~~0.625~~0.375]' \in\Lambda$ ensures that  $x_e=[7.6202~~29.9719~~200.0034~~0]'$ is a singular equilibrium point for this system, which is associated with a rotational velocity of $\omega_e\approx200$ rad/s. Conditions of Theorem~\ref{theo:common} are satisfied for this system considering 
	\begin{equation}
	\bar{P} = \begin{bmatrix}
	0.0142 & 0.0057 & 0.0068 \\ 
	0.0057 & 0.0108 & 0.0027 \\ 
	0.0068 & 0.0027 & 0.0048
	\end{bmatrix},~	
	P_\perp = 18.7476.
	\end{equation}
	Applying a piecewise constant disturbance $\tau(t)$ to the system evolving from $x(0)=0$, we obtained the trajectories given in Figure~\ref{fig:state_motor_vel}.
	
	\begin{figure}
		\centering
		\def\svgwidth{.55\linewidth}
		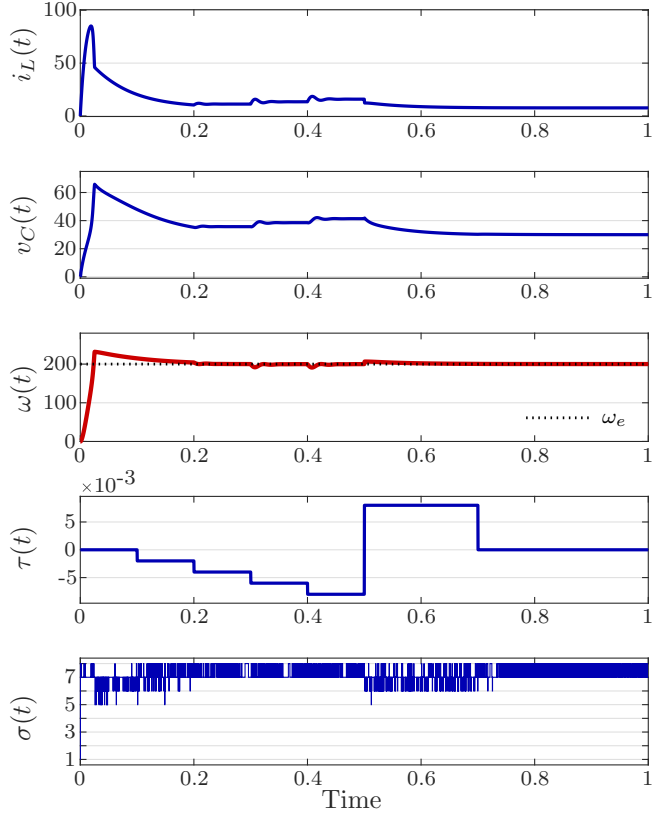
		\caption{Velocity control (integral action): State trajectories $i_L(t),~\nu_C(t)$ and $\omega(t)$ of the converter and the dc motor and switching signal $\sigma(t)$ generated by the state-dependent switching function~\eqref{eq:switch_fun}. The velocity reference $\omega_e=200$ rad/s is constant and the external disturbance $\tau(t)$ varies according to the profile.}\label{fig:state_motor_vel}
	\end{figure}
	These figures indicate that the proposed switching function may be used to incorporate integral action to the switching control of switched affine systems, which rejects the given external disturbance. Once more, we highlight that this is an empirical investigation and further analyses should be subject to future work, particularly on the estimation of bounds on the admissible disturbances.
	
	\section{Conclusion}
	Design conditions based on LMIs for switching function that globally asymptotically stabilizes singular equilibrium points of switched affine systems were developed in this work. The connections of {our results} with preexisting ones were discussed. In summary, we studied the case where the desired equilibrium point is the equilibrium of a convex combination of the dynamic {matrices that is rank-deficient}. A procedure based on SOS optimization was presented to estimate a local exponential convergence rate. Academical examples illustrated theoretical aspects of the theory and a practical application was presented, where the position and velocity control of a dc-motor fed by a boost converter and an h-bridge was carried out.

	\begin{ack}
		This research was supported by the ``National Council for Scientific and
		Technological Development (CNPq)'', under grant 303499/2018-4.
		
		R. Jungers  is a FNRS honorary Research Associate. This project has received funding from the European Research Council (ERC) under the \emph{European Union's Horizon 2020 research and innovation programme} under grant agreement No 864017 - L2C. RJ is also supported by the Innoviris Foundation and the FNRS (Chist-Era Druid-net).         
		
		For the SOS parser~\cite{Lofberg2009}, we thank the YALMIP team.
	\end{ack}
	
	\bibliographystyle{IEEEtran}
	\bibliography{refs}
	
	\appendix
	\section{Proof of Lemma~\ref{lemma:defective}}
	\label{app:proof_defective}
	\begin{proof}
		Consider the decomposition $X= \Vb\Vb'X +\Vp\Vp'X$, which leads to the equality
		\begin{equation}
		X^2 = \Vb\Vb'X\Vb\Vb'X +\Vp\Vp'X\Vb\Vb'X .\label{eq:proof_lemma_X2}
		\end{equation} 
		First, let us show that statement \textit{(b)} is equivalent to say that $\rank{(X^2)}<\rank{(X)}$. Indeed, let the Jordan Canonical Form $X=TJT^{-1}$ which implies that $\rank(J)= \rank(X)$ and $\rank(J^2)= \rank(X^2)$ given that $T$ is regular. According to~\cite[p. 168]{horn2012matrix}, $\rank{(X^2)}\leq\rank{(X)}-1$ if and only if at least one of the Jordan blocks has {dimension greater} than $1$ and is nilpotent, i.e., is associated with null eigenvalue. This is the case of $X$ due to the validity of statement \textit{(b)} and, therefore,  $\rank{(X^2)}<\rank{(X)}$.\\
		
		Now, let us first prove \textit{(a)}$\Rightarrow$\textit{(b)}. The singularity of $\Vb' X \Vb$ implies that there exists a non-zero $q\in\R^{p}$ such that $\Vb' X \Vb q=0$. Hence, the nullspace of $X^2$ is spanned by the columns of $\Vp$ together with the orthogonal vector $\Vb q$, which can verify in~\eqref{eq:proof_lemma_X2} that $X^2\Vb q=0$. Thus,  $\rank{(X^2)}<\rank{(X)}$ and $X$ is defective with $0$ as a defective eigenvalue.
		
		To demonstrate that \textit{(b)}$\Rightarrow$\textit{(a)}, assume $\rank{(X^2)}<\rank{(X)}$. Therefore, additionally to the nullspace spanned by $\Vp$, there must exist a non-zero orthogonal vector, for instance, $\Vb q $ with $q\in\R^{p}$ such that $X^2\Vb q= 0$. From~\eqref{eq:proof_lemma_X2}, we obtain that $X^2\Vb q= 0$ implies that 
		\begin{equation}
		\bar{q}:=(\Vb'X\Vb)^2 q =0 \quad \text{and} \quad q_\perp:=\Vp'X\Vb\Vb'X\Vb q= 0
		\end{equation}
		given that $X^2\Vb q=\Vb\bar{q} + \Vp q_\perp $ and that these two terms are orthogonal to each other. However, $\bar{q}=0$ implies that  $\Vb'X\Vb$ is singular. This concludes the proof.
	\end{proof}

		\section{Nonexistence of global exponential convergence rate} \label{app:proof_violated}
	
	Consider the error system~\eqref{eq:err_sy} having a shared nullspace $\mathcal{N}$, as considered in Theorem~\ref{theo:common}, and with affine terms $\ell_i\in\mathcal{N},~\forall i\in\K$. For this system, whenever the initial condition is taken as $\xi_0\in\mathcal{N}$, the solution $\xi(t)$  remains inside the nullspace $\mathcal{N}$ for all $t\geq 0$, given that $A_{\sigma(t)}\xi(t) =0$. Therefore, the right derivative of the state trajectory $\partial_{+} \xi(t)$ has its euclidean norm bounded by the affine terms as $\|\partial_{t+} \xi(t)\|\leq \max_{i\in\K} \|\ell_i\| =:\bar \ell$. By definition (see \cite{khalil2002nonlinear}), assuming global exponential stability implies that there must exists $c\geq1$ and $\alpha>0$ such that
	\begin{equation}
	\|\xi(t)\|\leq c\|\xi_0\|e^{-\alpha t}\label{eq:exp_stab_def}
	\end{equation}
	holds for all $\xi_0\in\R^n$ and $t\geq0$. Thus, we have
	\begin{align}
	0&\leq c\|\xi_0\|e^{-\alpha t}-\|\xi(t)\|\nonumber\\
	&= c\|\xi_0\|e^{-\alpha t}-\left\|\xi_0 + \int_{0}^{t}\partial_{+} \xi(\tau)d\tau\right\|\nonumber\\
	&\leq c\|\xi_0\|e^{-\alpha t}-\|\xi_0\|+\left\|\int_{0}^{t}\partial_{+} \xi(\tau)d\tau\right\|\nonumber\\		
	&\leq \|\xi_0\|(ce^{-\alpha t}-1)+ \int_{0}^{t}\left\|\partial_{+} \xi(\tau)\right\|d\tau\nonumber\\		
	&\leq  \|\xi_0\|(ce^{-\alpha t}-1)+\bar\ell t\label{eq:app_last_ineq_viol}
	\end{align}
	where the second inequality follows from the reverse triangle inequality.	However, for any scalars $c\geq 1,~\alpha>0$ and $\bar\ell>0$, one can always find finite values for $\|\xi_0\|>0$ and $t>0$ such that \eqref{eq:app_last_ineq_viol} is violated. For instance, we can easily verify that $\|\xi_0\|(ce^{-\alpha t}-1)+\bar\ell t= -1$ for $t=\ln(2c)/\alpha$ and any $\xi_0$ such that $\|\xi_0\|= 2(\bar\ell\ln(2c)/\alpha +1)$. Therefore, there cannot exist  $c\geq1$ and $\alpha>0$ for which~\eqref{eq:exp_stab_def} holds for all $\xi_0\in\R^{n}$ and $t\geq 0$.

	\section{Proof of Corollary~\ref{coro:coro01}}\label{app:B}
	\begin{proof}
		Assume that the conditions in Theorem~\ref{theo:common} hold. Therefore, the strict inequality in~\eqref{eq:lmi1} implies that there exist a scalar $\epsilon>0$ small enough  and a matrix $G_\epsilon>0$ such that~\eqref{eq:gamma_eps_plus}-\eqref{eq:gamma_eps_minus} hold for  all $i\in\K_a(\lambda)$ and $\gamma(\lambda,\epsilon,i),\gamma(\lambda,-\epsilon,i)\in \Lambda$. Thus, for each $i\in\K_a(\lambda)$, the upper bound~\eqref{eq:upper_lim_v_dot} evaluated for $\mu=\gamma(\lambda,\epsilon,i)$ and $\mu=\gamma(\lambda,-\epsilon,i)$ yields two upper bounds
		\begin{align}
		\dot{v}(\xi)& \leq (1-\epsilon)f_\lambda(\xi) +\epsilon f_i(\xi)\\
		\dot{v}(\xi)& \leq (1+\epsilon)f_\lambda(\xi) -\epsilon f_i(\xi)
		\end{align}
		which, recalling the definition of $f_i(\xi)$ in~\eqref{eq:v_dot}, are respectively equivalent to
		\begin{align}
		\dot{v}(\xi)& \leq 2\bar{\xi}'\left((\bar{P}\Vb' +{{P}_\times} {\Vp'})A_{\gamma(\lambda,i,\epsilon)} \Vb \right)\bar{\xi} +2\epsilon g_i(\xi)\nonumber\\
		\dot{v}(\xi)& \leq 2\bar{\xi}'\left((\bar{P}\Vb' +{{P}_\times} {\Vp'})A_{\gamma(\lambda,i,-\epsilon)} \Vb \right)\bar{\xi} -2\epsilon g_i(\xi)\nonumber
		\end{align}
		with $g_i(\xi)$ given in~\eqref{eq:g_i}, where $A_\lambda \Vp=0$, $U_\lambda=0$ and $\ell_\lambda=A_\lambda x_e +b_\lambda=0$ were once more used. 
		Combining these inequalities together and recalling~\eqref{eq:gamma_eps_plus}-\eqref{eq:gamma_eps_minus}, we have
		\begin{equation}\label{eq:abs}
		\dot{v}(\xi) \leq -\bar\xi'G_\epsilon\bar\xi -2\epsilon |g_i(\xi)|.
		\end{equation}
		Notice that, as a property of the absolute value function, the inequality
		\begin{equation}
		|g_i(\xi)| \geq  \beta \big(g_i(\xi)\big)^2
		\end{equation}
		holds for all $\xi$ such that $|g_i(\xi)|\leq1/\beta$ with $\beta>0$. 
		Also, the Taylor series expansion of $\xi \mapsto\big(g_i(\xi)\big)^2$ around $\xi=0$ yields
		\begin{equation}
		\big(g_i(\xi)\big)^2 = h_i(\xi)+ O^3(\xi)
		\end{equation}
		with the quadratic $h_i(\xi)$ given in~\eqref{eq:h_i} and  $O^3(\xi)$ denotes a sum of terms of $\xi$ with total order equal to or greater than three. Therefore, there exists a bounded region $\tilde \Omega_i$, defined in~\eqref{eq:omega_tilde} by all $\xi\in\R^n$ such that
		\begin{equation}\label{eq:gi_hi}
		|g_i(\xi)| \geq\beta h_i(\xi)
		\end{equation}
		for which the inequality~\eqref{eq:abs} yields
		\begin{equation}
		\dot v(\xi)\leq -\bar{\xi}'G_\epsilon\bar{\xi} - 2\epsilon\beta h_i(\xi) . \label{eq:v_upperbound_cost}
		\end{equation}
		The region $\tilde{\Omega}_i$  is non-empty because $g_i(0)=0$ and $h_i(0)=0$, which implies that it contains the origin $\xi=0$. Also, this region is not a singleton because {the inequality~\eqref{eq:gi_hi} that defines $\tilde{\Omega}_i$} is satisfied for any point $\xi$ arbitrarily close to the origin, given that the linear terms of the left-hand side dominate the quadratic one in the right-hand side. In special, when $\xi$ is such that {the} linear terms of $g_i(\xi )$ yield $\bar{\xi}'\bar{S}\ell_i +\xi_{\perp}'S_\perp\ell_i=0$, the right-hand side of the inequality is also null, given the definition~\eqref{eq:h_i}.
		
		Now, let a sublevel set $\Omega$ of the Lyapunov function $v(\xi)$ be included in $\mathcal{X}$, which is defined in~\eqref{eq:region_omega}. By definition, the upper bound~\eqref{eq:v_upperbound_cost} holds for all $i\in\K_a(\lambda)$ inside $\mathcal{X}$. Therefore, we have 
		\begin{align}
		\dot v(\xi(t))
		&\leq \!-\bar{\xi}'G_\epsilon\bar{\xi}\! -\! \sum_{i\in \K_a(\lambda)}\frac{2\epsilon\beta}{|\K_a(\lambda)|} h_i(\xi)\nonumber\\
		&=\!\!-\!\!\begin{bmatrix}
		\bar{\xi} \\ \xi_\perp
		\end{bmatrix}'\!\!\!\left(\begin{bmatrix}
		G_\epsilon&0\\
		0&0
		\end{bmatrix}\!+\! \frac{\epsilon\beta}{|\K_a(\lambda)|}\!\begin{bmatrix}
		\bar{S}L\\S_\perp L
		\end{bmatrix}\!\! \begin{bmatrix}
		\bar{S}L\\S_\perp L
		\end{bmatrix}'\right)\!\!\begin{bmatrix}
		\bar{\xi} \\ \xi_\perp
		\end{bmatrix} \nonumber\\
		&<0\label{eq:v_quad_local_upperbound}
		\end{align}
		where the first inequality averages the upper bounds in~\eqref{eq:v_upperbound_cost} for all $i\in\K_a(\lambda)$. Finally, the last one holds for all nonzero $(\bar{\xi},\xi_\perp)$ as for $\bar{\xi}\neq0$ the first augmented matrix yields $\bar{\xi}'G_\epsilon\bar{\xi}>0 $ and the second one, a non-negative value, whereas for $\bar{\xi}=0$, the second matrix yields $\xi_\perp'S_\perp LL'S_\perp'\xi_\perp>0$, which is a consequence from the fact that $S_\perp L = P_\perp ML$ has full rank. This has been ensured by Theorem~\ref{theo:common}{,~through the condition $0\in{\rm Int}(\Pp)$}. Hence, we demonstrate, by the comparison lemma~\cite[p.~102]{khalil2002nonlinear} that local exponential stability holds inside $\Omega$. 
	\end{proof}
	\cv
\end{document}